\documentclass[12pt]{amsart}
\usepackage{amsfonts,amsthm,amsmath,amssymb,verbatim}
\usepackage{graphicx}
\usepackage[active]{srcltx}
\usepackage[all,cmtip]{xy}

\newtheorem{thm}{Theorem}[section]
\newtheorem{prop}[thm]{Proposition}
\newtheorem{lemma}[thm]{Lemma}

\newtheorem{cor}[thm]{Corollary}
\newtheorem{conj}[thm]{Conjecture}

\theoremstyle{remark}

\newtheorem{example}[thm]{Example}
\newtheorem{remark}[thm]{Remark}
\newtheorem{defin}{Definition}



\def\C{\mathbb{C}}
\def\R{\mathbb{R}}
\def\Z{\mathbb{Z}}
\def\P{\mathbb{P}}

\def\I{\mathbb{I}}

\def\Pic{{\rm Pic}}

\def\Oc{\mathcal{O}}
\def\a{\alpha}

\def\l{\lambda}

\def\emptyset{\varnothing}

\title{Divided difference operators on polytopes}
\author{Valentina Kiritchenko}
\email{vkiritch@hse.ru}
\thanks{The author was supported by Dynasty foundation, AG Laboratory NRU HSE,
MESRF grants ag. 11.G34.31.0023, MK-983.2013.1 and by RFBR grants 12-01-31429-mol-a, 12-01-33101-mol-a-ved.
This study was carried out within ``The National Research
University Higher School of Economics'  Academic Fund Program in 2013-2014,
research grant No. 12-01-0194''}

\address{Laboratory of Algebraic Geometry and Faculty of Mathematics\\ Higher School of Economics\\
Vavilova St. 7, 112312 Moscow, Russia}
\address{Institute for Information Transmission Problems, Moscow, Russia}

\date{}

\begin{document}
\begin{abstract}
We define convex-geometric counterparts of divided difference (or Demazure) operators from the Schubert calculus and representation theory.
These operators are used to construct inductively polytopes that capture Demazure characters of representations of reductive groups.
In particular, Gelfand--Zetlin polytopes and twisted cubes of Grossberg--Karshon are obtained in a uniform way.
\end{abstract}

\maketitle

\section{Introduction}
Polytopes play a prominent role in representation theory and algebraic geometry.
In algebraic geometry, there are Okounkov convex bodies introduced by Kaveh--Khovanskii and Lazarsfeld--Mustata (see \cite{KKh} for the references).
These convex bodies turn out to be polytopes in many important cases (e.g. for spherical varieties).
In representation theory, there are string polytopes introduced by Berenstein--Zelevinsky and Littelmann \cite{BZ,L}.
String polytopes are associated with the irreducible representations of a reductive group $G$, namely, the integer points inside and at the boundary of a string polytope parameterize a canonical basis in the corresponding representation.
A classical example of a string polytope for $G=GL_n$ is a Gelfand--Zetlin polytope.

There is a close relationship between string polytopes and Okounkov bodies.
String polytopes were identified with Okounkov polytopes of flag varieties for a certain choice of a geometric valuation \cite{K} and were also used in \cite{KKh} to give a more explicit description of Okounkov bodies associated with actions of $G$ on algebraic varieties.
Natural generalizations of string polytopes are Okounkov polytopes of Bott--Samelson resolutions of Schubert varieties for various geometric valuations (an example of such a polytope is computed in \cite{Anderson}).

In this paper, we introduce an elementary convex-geometric construction that yields polytopes with the same properties as string polytopes and Okounkov polytopes of Bott--Samelson resolutions.
Namely, exponential sums over the integer points inside these  polytopes yield Demazure characters.
We start from a single point and apply a sequence of simple convex-geometric operators that mimic the well-known {\em divided difference} or {\em Demazure operators} from the Schubert calculus and representation theory.
Convex-geometric Demazure operators act on convex polytopes and take a polytope to a polytope of dimension one greater.
In particular, classical Gelfand--Zetlin polytopes can be obtained in this way (see Section \ref{ss.GL}).
More generally, these operators act on {\em convex chains}.
The latter were defined and studied in \cite{PKh} and used in \cite{PKh2} to prove a convex-geometric variant of the Riemann--Roch theorem.

When $G=GL_n$, convex-geometric Demazure  operators were implicitly used in \cite{KST} to calculate Demazure characters of Schubert varieties in terms of the exponential sums over unions of faces of Gelfand--Zetlin polytopes and to represent Schubert cycles by unions of faces.
A motivation for the present paper is to create a general framework for extending results of \cite{KST} on Schubert calculus from type $A$ to arbitrary reductive groups.
In particular, convex-geometric divided difference operators allow one to use in all types a geometric version of mitosis (mitosis on parallelepipeds) developed in \cite[Section 6]{KST}.
This might help to find an analog of mitosis of \cite{KnM} in other types.

Another motivation is to give a tool for describing inductively Okounkov polytopes of Bott--Samelson resolutions.
We describe polytopes that conjecturally coincide with Okounkov polytopes of Bott--Samelson resolutions for a natural choice of a geometric valuation (see Conjecture \ref{conj}).
Another application is an inductive description of Newton--Okounkov polytopes for line bundles on Bott towers (in particular, on toric degenerations of Bott--Samelson resolutions) that were first described by Grossberg and  Karshon \cite{GK} (see Section \ref{ss.Bott} and  Remark \ref{r.GK}).

This paper is organized as follows. 
In Section \ref{s.main}, we give background on convex chains and define convex-geometric divided difference operators. 
In Section \ref{s.Demazure}, we relate these operators with Demazure characters and their generalizations. 
In Section \ref{s.ag}, we outline possible applications to Okounkov polytopes of Bott towers and Bott--Samelson varieties.

I am grateful to Dave Anderson, Joel Kamnitzer, Kiumars Kaveh and Askold Khovanskii for useful discussions.

\section{Main construction} \label{s.main}

\subsection{String spaces and parapolytopes}
\begin{defin} A {\em string space} of {\em rank $r$} is a real vector space ${\mathbb R}^d$
together with a direct sum decomposition
$${\mathbb R}^d={\mathbb R}^{d_1}\oplus\ldots\oplus{\mathbb R}^{d_r} $$
and a collection of linear functions $l_1$, \ldots, $l_r\in({\mathbb R}^d)^*$
such that $l_i$ vanishes on ${\mathbb R}^{d_i}$.
\end{defin}

We choose coordinates in $\mathbb R^d$ such that they are compatible with the direct sum
decomposition.
The coordinates will be denoted by
$(x_1^1,\ldots,x_{d_1}^1;\ldots; x_1^r,\ldots, x_{d_r}^r)$
so that the summand $\R^{d_i}$ is given by vanishing of all coordinates except for $x_1^i$,\ldots, $x_{d_i}^i$.
In what follows, we regard $\R^d$ as an affine space.

Let $\mu=(\mu_1,\dots,\mu_{d_i})$ and $\nu=(\nu_1,\dots,\nu_{d_i})$ be two collections of real numbers such that $\mu_j\le\nu_j$ for all $j=1$,\dots, $d_i$.
By the {\em coordinate parallelepiped} $\Pi(\mu,\nu)\subset\R^{d_i} $ we mean the parallelepiped
$$
\Pi(\mu,\nu)=\{(x^i_1,\ldots,x^i_{d_i})\in\R^{d_i}\ | \ \mu_j\le x_j^{i}\le \nu_j, \ j=1,\dots, d_i\}.
$$

\begin{defin}
A convex polytope $P\subset{\mathbb R}^d$ is called a {\em parapolytope} if for $i=1$,\ldots, $r$, and a vector $c\in\R^d$ the intersection of $P$
with the parallel translate $c+{\mathbb R}^{d_i}$ of $\R^{d_i}$ is the parallel translate of a coordinate parallelepiped, i.e.,
$$P\cap(c+\R^{d_i})=c+\Pi(\mu_c,\nu_c)$$
for $\mu_c$ and $\nu_c$ that depend on $c$.
\end{defin}

For instance, if $d=r$ (i.e., $d_1=\ldots=d_r=1$) then every polytope is a
parapolytope.
Below is a less trivial example of a parapolytope in a string space.
\begin{example}\label{e.GZ} Consider the string space
$$\mathbb R^d=\mathbb R^{n-1}\oplus\mathbb R^{n-2}\oplus\ldots\oplus\mathbb R^1$$
of rank $r=(n-1)$ and dimension $d=\frac{n(n-1)}2$.

Let $\lambda=(\lambda_1,\ldots,\lambda_n)$ be a non-increasing collection of integers.
For each $\l$, define the {\em Gelfand--Zetlin polytope} $Q_\l$ by the inequalities
$$
\begin{array}{cccccccccc}
\l_1&       & \l_2    &         &\l_3          &    &\ldots    & &       &\l_n   \\
    &x^1_1&         &x^1_2  &         & \ldots   &       &  &x^1_{n-1}&       \\
    &       &x^2_1 &       &  \ldots &   &        &x^2_{n-2}&         &       \\
    &       &       &  \ddots   & &  \ddots   &      &         &         &       \\
    &       &       &  &x^{n-2}_1&     &  x^{n-2}_2 &        &         &       \\
    &       &         &    &     &x^{n-1}_1&   &              &         &       \\
\end{array}
$$
where the notation
$$
 \begin{array}{ccc}
  a &  &b \\
   & c &
 \end{array}
 $$
means $a\ge c\ge b$.
It is easy to check that $Q_\l$ is a parapolytope.
Indeed, consider a parallel translate of $\R^{n-i}$ by the vector $c=(c^1_1,\ldots,c^1_{n-1};\ldots;c^{n-1}_1)$.
Put $c^0_i=\l_i$ for $i=1$,\ldots, $n$.
The intersection of $Q_\l$  with $c+\R^{n-i}$ is given by the
the following inequalities:
$$
\begin{array}{ccccccccc}
c^{i-1}_1&   &c^{i-1}_2&   &c^{i-1}_3&       &\ldots&   &c^{i-1}_{n-i+1}\\
     &x^i_1&     &x^i_2&     &\ldots &      &x^i_{n-i}&         \\
     &   &c^{i+1}_1&   &\ldots&      &c^{i+1}_{n-i-1}&  &       \\
\end{array}.
$$
Therefore, the intersection can be identified with the coordinate parallelepiped
$c+\Pi(\mu,\nu)\subset c+\R^{n-i}$, where $\mu_j=\max(c^{i-1}_{j},c^{i+1}_{j-1})$ and $\nu_j=\min(c^{i-1}_{j+1},c^{i+1}_{j})$
(put $c^{i+1}_0=-\infty$ and $c^{i+1}_{n-i}=+\infty$).
\end{example}

\subsection{Polytopes and convex chains}
Consider the set of all convex polytopes in $\R^d$.
This set can be endowed with the structure of a commutative semigroup using {\em Minkowski sum}
$$
P_1+P_2=\{x_1+x_2\in\R^d\ |\ x_1\in P_1,\ x_2\in P_2\}
$$
It is not hard to check that this semigroup has cancelation property.
We can also multiply polytopes by positive real numbers using dilation:
$$
\lambda P=\{\lambda x\ |\ x\in P\},\quad \lambda\ge 0.
$$
Hence, we can embed the semigroup $S$ of convex polytopes into its Grothendieck group $V$,
which is a real (infinite-dimensional) vector space.
The elements of $V$ are called {\em virtual polytopes}.

It is easy to check that the set of parapolytopes in $\R^d$ is closed under Minkowski sum
and under dilations.
Hence, we can define the subspace $V_{\square}\subset V$ of {\em virtual parapolytopes} in the string space $\R^d$.

\begin{example} If $\R^d$ is a string space of rank $1$, i.e. $d_1=d$, then parapolytopes are coordinate parallelepipeds $\Pi(\mu,\nu)$.
Clearly,
$$\Pi(\mu,\nu)+\Pi(\mu',\nu')=\Pi(\mu+\mu',\nu+\nu').$$
Hence, virtual parapolytopes can be identified with the pairs of vectors $\mu,\nu\in \R^d$.
This yields an isomorphism $V_{\square}\simeq \R^d\oplus\R^d$.
Under this isomorphism, the semigroup of (true) coordinate parallelepipeds gets mapped to the convex cone in $\R^d\oplus\R^d$
given by the inequalities $\mu_i\le\nu_i$ for $i=1$,\ldots, $d$.
\end{example}

We now define the space $\tilde V$ of {\em convex chains} following \cite{PKh}.
A {\em convex chain} is a function on $\R^d$ that can be represented as a finite linear combination
$$\sum_P c_P {\I}_P,$$
where $c_P\in\R$, and ${\I}_P$ is the characteristic function of a convex polytope $P\subset\R^d$, that is,
$${\I}_P(x)=\left\{\begin{array}{cc}
                  1, & x\in P \\
                  0, & x\notin P
                \end{array}
\right. .$$
The semigroup $S$ of convex polytopes can be naturally embedded into $\tilde V$:
$$\iota:S\hookrightarrow\tilde V;\quad \iota:P\mapsto  {\I}_P$$
In what follows, we will work in the space of convex chains and freely identify a polytope $P$ with the corresponding convex chain $\I_P$.
However, note that the embedding $\iota$ is not a homomorphism, that is, ${\I}_{P+Q}\ne{\I}_P+{\I}_Q$ (the sum of convex chains is defined as the usual sum of functions).

\begin{remark} The embedding $\iota:S\hookrightarrow\tilde V$ can be extended to the space $V$ of all virtual polytopes.
Namely, there exists a commutative operation $*$ on $\tilde V$ (called {\em product of convex chains}) such that
$${\I}_{P+Q}={\I}_P*{\I}_Q \eqno(M)$$
for any two convex polytopes $P$ and $Q$
(see \cite[Section 2, Proposition-Definition 3]{PKh}).
Virtual polytopes can be identified with the convex chains that are invertible with respect to $*$.
\end{remark}

Similarly to the space of convex chains, define the subspace $\tilde V_{\square}\subset \tilde V$ of {\em convex parachains} using only parapolytopes instead of all polytopes.
We will use repeatedly the following example of a parachain.

\begin{example} \label{e.paral_chain}
Consider the simplest case $d=1$.
Let $[\mu,\nu]\subset\R$ be a segment (i.e., $\mu<\nu$), and $[\nu,\mu]$ --- a virtual segment.
Using the existence of the operation $*$ satisfying $(M)$, it is easy to check that
$$\iota([\nu,\mu])=-\I_{[-\nu,-\mu]}+\I_{\{-\nu\}}+\I_{\{-\mu\}}$$
(note that the right hand side is the characteristic function of the open interval $(-\nu,-\mu)$).
Indeed,
$$\I_{[\mu,\nu]}*\left(-\I_{[-\nu,-\mu]}+\I_{\{-\nu\}}+\I_{\{-\mu\}}\right)=
-\I_{[\mu,\nu]}\ast\I_{[-\nu,-\mu]}+\I_{[\nu,\mu]}\ast\I_{\{-\nu\}}+
\I_{[\mu,\nu]}\ast\I_{\{-\mu\}}$$
$$=-\I_{[\mu-\nu,\nu-\mu]}+\I_{[\mu-\nu,0]}+\I_{[0,\nu-\mu]}=\I_{\{0\}}.$$
\end{example}
More generally, if $P\subset\R^d$ is a convex polytope then
$$(-1)^{\dim P}\I_P*\I_{{\rm{int}}(P^\vee)}=\I_{\{0\}},$$
where $P^\vee=\{-x\ |\ x\in P\}$, and ${\rm{int}}(P^\vee)$ denotes the interior of $P^\vee$ (see \cite[Section 2, Theorem 2]{PKh}).

\subsection{Divided difference operators on parachains}
For each $i=1$,\ldots, $n$, we now define a {\em divided difference (or Demazure) operator}
$D_i$ on the space of convex parachains $\tilde V_{\square}$.
It is enough to define $D_i$ on convex parapolytopes and then extend the definition by linearity to the other parachains.
Let $P$ be a parapolytope.
Choose the smallest $j=1$,\ldots, $d_i$ such that $P$ lies in the hyperplane $\{x^i_j=\mathrm{const}\}$.
If no such $j$ exists, then $D_i(\I_P)$ is not defined.
Otherwise, we expand $P$ in the direction of $x^i_j$ as follows.

First, suppose that a parapolytope $P$ lies in $(c+{\mathbb R}^{d_i})$ for some
$c\in{\mathbb R}^{d}$, i.e., $P=c+\Pi(\mu,\nu)$ is a coordinate parallelepiped.
We always fix the choice of $c$ by requiring that $c$ lies in the direct complement to $\R^{d_i}$ with respect to the decomposition
$\R^d=\R^{d_1}\oplus\ldots\oplus\R^{d_i}\oplus\ldots\oplus\R^{d_r}$.
Consider $\nu'=(\nu_1',\ldots,\nu_{d_i}')$, where $\nu'_k=\nu_k$ for all $k\ne j$, and $\nu'_j$ is defined by the equality
$$\sum_{k=1}^{d_i}(\mu_k+\nu'_k)=l_i(c).$$
If $\nu'_j\ge\nu_j$, then $D^+_i(P):=c+\Pi(\mu,\nu')$ is a true coordinate parallelepiped.
Note that $P$ is a facet of $D^+_i(P)$ unless $\nu'=\nu$

If $\nu'_j<\nu_j$, define $\mu'=(\mu_1',\ldots,\mu_{d_i}')$ by setting $\mu'_k=\mu_k$ for all $k\ne j$, and $\mu'_j=\nu_j'$.
Then $D^-_i(P):=c+\Pi(\mu',\nu)$ is a true coordinate parallelepiped, and $P$ is a facet of
$D^-_i(P)$.
Let $P'$ be the facet of $D^-_i(P)$ parallel to $P$.

We now define $D_i(\I_P)$ as follows:
$$D_i(\I_P)=\left\{
\begin{array}{ll}
\I_{D^+_i(P)} & \mbox{ if } \nu_j\le\nu'_j, \\
-\I_{D^-_i(P)}+\I_{P}+\I_{P'} & \mbox{ if } \nu_j>\nu'_j. \\
\end{array}
\right.$$

\begin{remark} \label{r.segment}
This definition is motivated by the following observation.
Let $\mu$ and $\nu$ be integers such that $\mu<\nu$.
Define the function $f(\mu,\nu,t)$ of a complex variable $t$ by the formula
$$f(\mu,\nu,t)=t^{\mu}+t^{\mu+1}+\ldots+t^{\nu},$$
that is, $f$ is the exponential sum over all integer points in the segment $[\mu,\nu]\subset\R$.
Computing the sum of the geometric progression, we get that
$$f(\mu,\nu,t)=\frac{t^\mu-t^{\nu+1}}{1-t}.$$
This formula gives a meromorphic continuation of $f(\mu,\nu,t)$ to all real $\mu$ and $\nu$.
In particular, for integer $\mu$ and $\nu$ such that $\mu>\nu$ we obtain
$$f(\mu,\nu,t)=\frac{t^\mu-t^{\nu+1}}{1-t}=-(t^{\nu+1}+\ldots+t^{\mu-1}),$$
that is, $f$ is minus the exponential sum over all integer points in the open interval $(\nu,\mu)\subset\R$ (cf. Example \ref{e.paral_chain}).
\end{remark}

\begin{defin}\label{d.main} For an arbitrary parapolytope $P\subset{\mathbb R}^d$
define $D_i(\I_P)$ by setting
$$D_i(\I_P)\left|_{c+\R^{d_i}}\right.=D_i(\I_{P\cap(c+{\mathbb R}^{d_i})})$$
for all $c$ in the complement to $\R^{d_i}$.
\end{defin}

It is not hard to check that this definition yields a convex chain.
In many cases (see examples below and in Section \ref{s.Demazure}), $D_i(\I_P)$ is the characteristic function of a polytope (and $P$ is a facet of this polytope unless $D_i(\I_P)=\I_P$).
This polytope will be denoted by $D_i(P)$.

The definition immediately implies that similarly to the classical Demazure operators the convex-geometric ones satisfy the identity $D_i^2=D_i$.
It would be interesting to find an analog of braid relations for these operators.

\subsection{Examples} \label{ss.examples}
\paragraph{\bf Dimension 2.} The simplest meaningful example is $\R^2=\R\oplus\R$.
Label coordinates in $\R^2$ by $x:=x_1^1$ and $y:=x_2^1$.
Assume that $l_1=y$ and $l_2=x$.
If $P=\{(\mu_1,\mu_2)\}$ is a point, and $\mu_2\ge 2\mu_1$, then $D_1(P)$ is a segment:
$$D_1(P)=[(\mu_1,\mu_2),(\mu_2-\mu_1,\mu_2)].$$
If $\mu_2<2\mu_1$, then $D_1(\I_P)$ is a virtual segment, that is,
$$D_1(\I_P)=-\I_{[(\mu_2-\mu_1,\mu_2),(\mu_1,\mu_2)]}+\I_P+\I_{(\mu_2-\mu_1,\mu_2)}.$$

If $P=AB$ is a horizontal segment, where $A=(\mu_1,\mu_2)$ and $B=(\nu_1,\mu_2)$, then $D_2(P)$ is the trapezoid $ABCD$ given by the inequalities
$$\mu_1\le x\le\nu_1,\quad \mu_2\le y\le x-\mu_2.$$
See Figure 1 for $D_2(P)$ in the case $\mu_1=-1$, $\nu_1=2$, $\mu_2=-1$ (left) and
$\mu_1=-1$, $\nu_1=2$, $\mu_2=0$ (right).
In the latter case, the convex chain $D_2(\I_P)$ is equal to
$$\I_{OBC}-\I_{ADO}+\I_{OA}+\I_{DO}-\I_{O}.$$
\begin{figure}
\includegraphics[width=7.25cm]{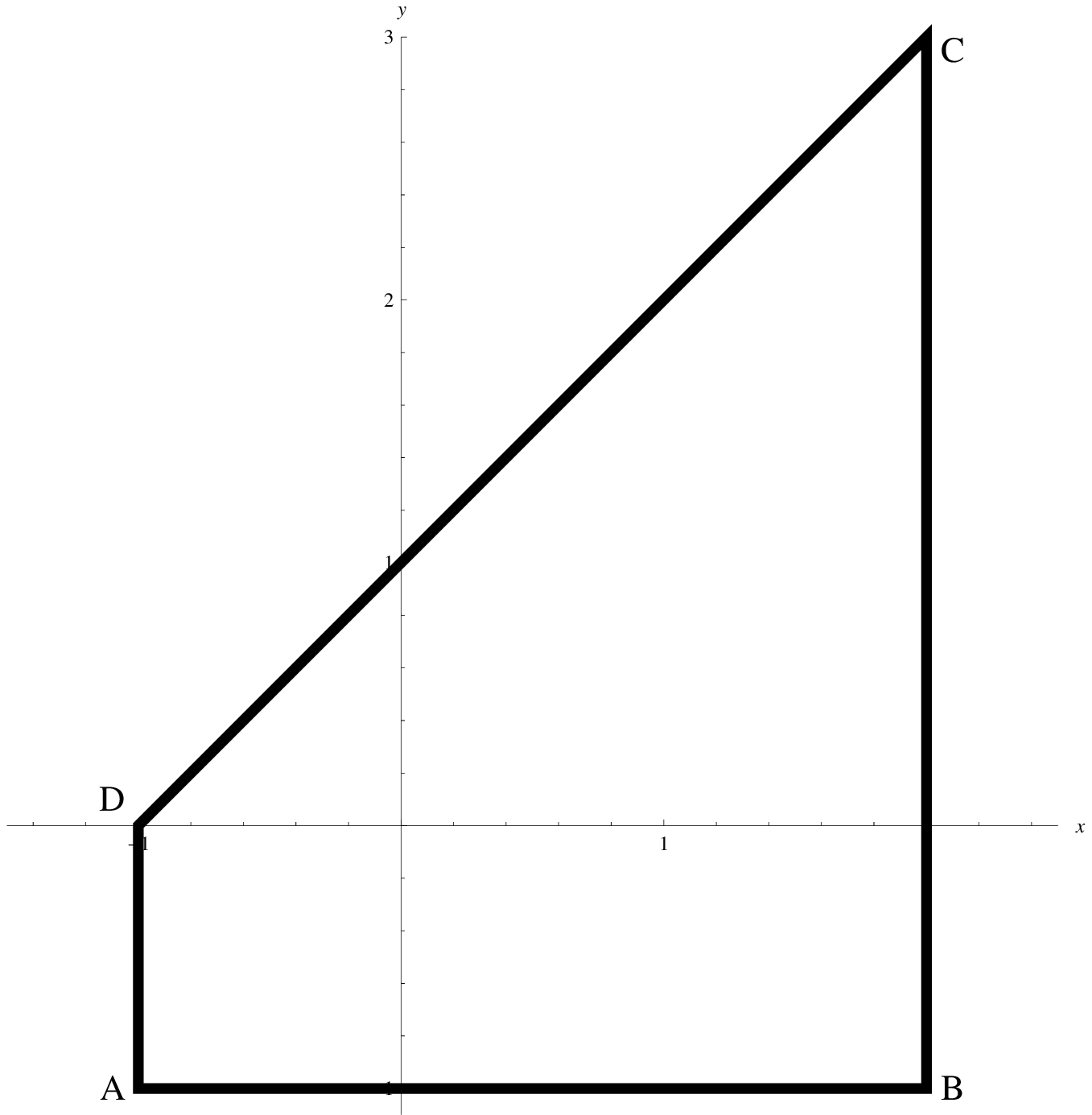} \includegraphics[width=7.25cm]{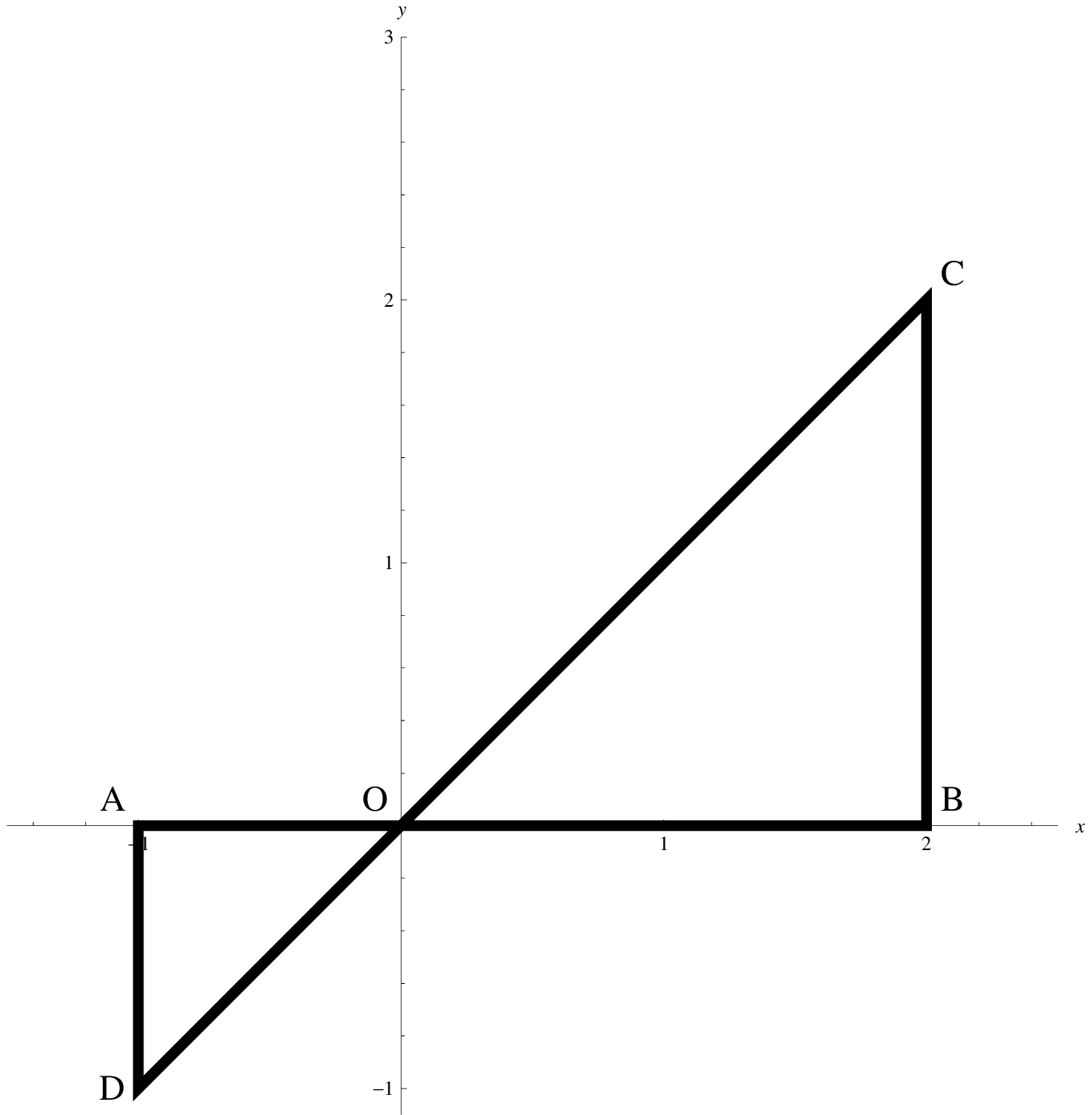}
\caption{Trapezoids $D_2(P)$ for different segments $P=AB$.}
\end{figure}

\paragraph{\bf Dimension 3.} A more interesting example is $\R^3=\R^2\oplus\R$.
Label coordinates in $\R^3$ by $x:=x_1^1$, $y:=x_2^1$ and $z:=x_1^2$.
Assume that $l_1=z$ and $l_2=x+y$.
If $P=(\mu_1,\mu_2,\mu_3)$ is a point, then $D_1(P)$ is a segment:
$$D_1(P)=[(\mu_1,\mu_2,\mu_3),(\mu_3-\mu_1-2\mu_2,\mu_2,\mu_3)].$$
Similarly, if $P=[(\mu_1,\mu_2,\mu_3),(\nu_1,\mu_2,\mu_3)]$ is a segment in $\R^2$, then $D_1(P)$ is the rectangle given by the equation $z=\mu_3$ and the inequalities
$$\mu_1\le x\le\nu_1,\quad \mu_2\le y\le \mu_3-\mu_1-\nu_1-\mu_2.$$

Using the previous calculations, it is easy to show
that if $P=(\l_2,\l_3,\l_3)$ is a point and $\l_3<\l_2<-\l_2-\l_3$, then $D_1D_2D_1(P)$ is the 3-dimensional Gelfand--Zetlin polytope $Q_\l$ (as defined in Example \ref{e.GZ}) for $\l=(\l_1,\l_2,\l_3)$, where $\l_1=-\l_2-\l_3$.
Indeed, $D_2D_1(P)$ is the trapezoid (see Figure 2) given by the equation $y=\l_3$ and the inequalities
$$\l_2\le x\le\l_1,\quad \l_3\le z\le x.$$
Then $D_1D_2D_1(P)$ is the union of all rectangles $D_2(I_a)$ for $a\in[\l_3,\l_1]$, where $I_a$ is the segment $D_2D_1(P)\cap\{z=a$\}, that is, $I_a=[(\max\{z,\l_2\},\l_3,a),(\l_1,\l_2,a)]$. Hence,
$$\l_3\le y\le \min\{\l_2,z\}.$$

\begin{figure}
\includegraphics[width=7cm]{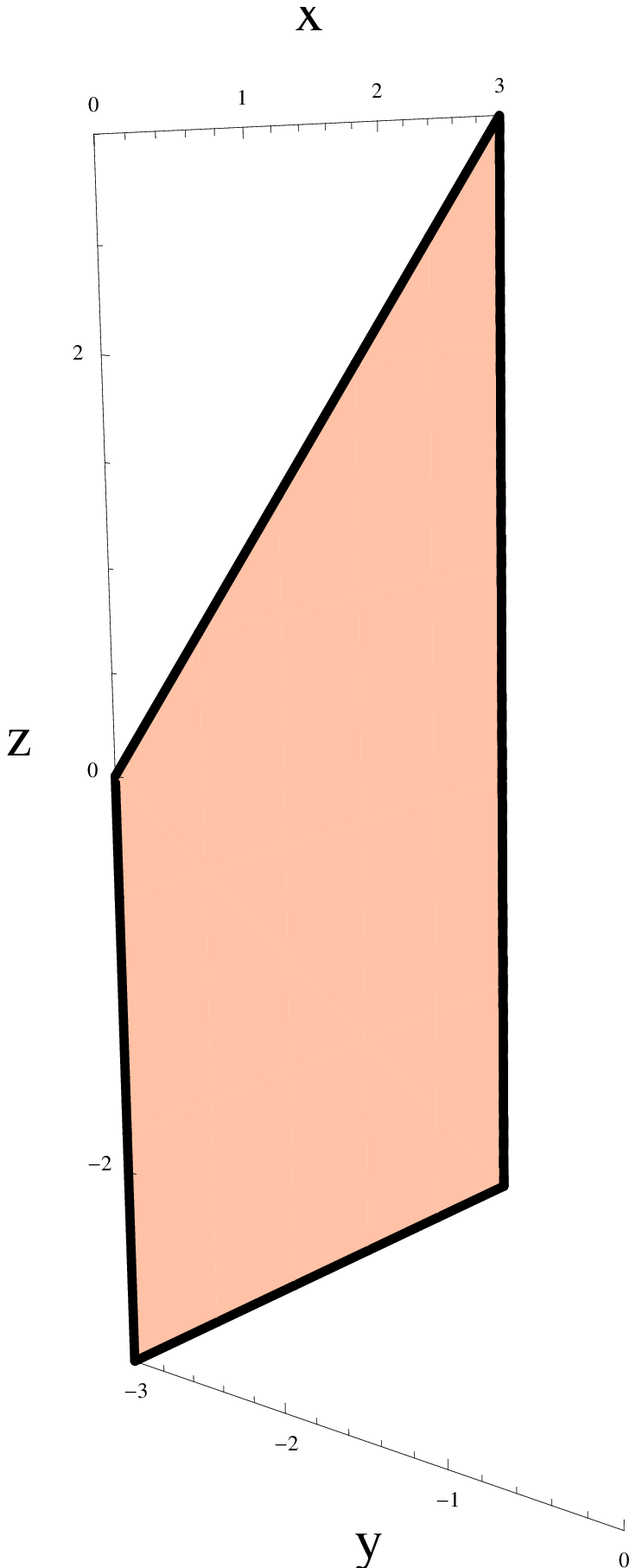} \includegraphics[width=7cm]{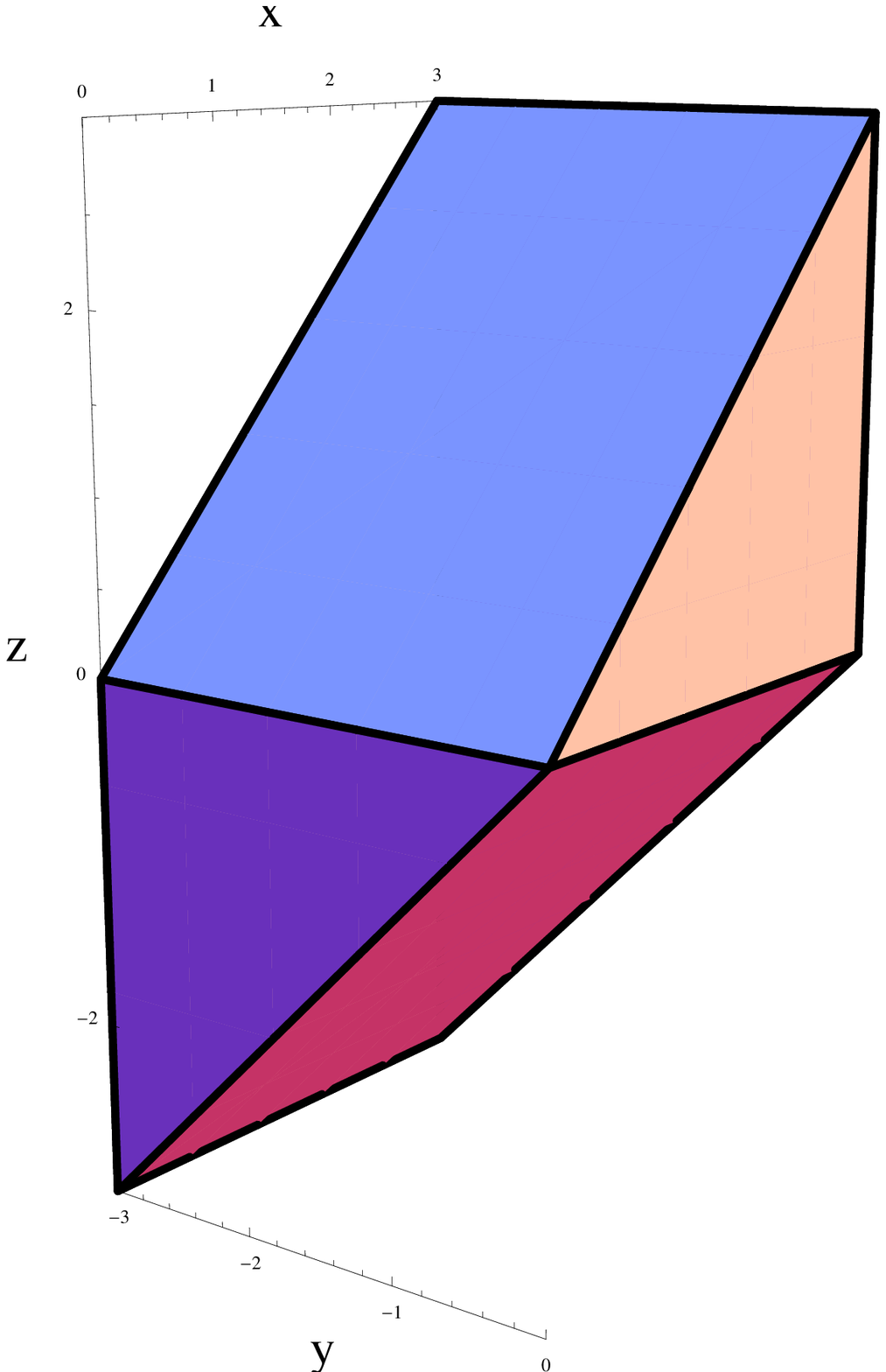}
\caption{Trapezoid $D_2D_1(P)$ and polytope $D_1D_2D_1(P)$ for a point $P=(0,-3,-3)$}
\end{figure}
Similarly to the last example, we construct Gelfand--Zetlin polytopes for
arbitrary $n$ using the string space from Example \ref{e.GZ} (see Theorem \ref{t.GZ}).

\section{Polytopes and Demazure characters}\label{s.Demazure}

\subsection{Characters of polytopes}
For a string space $\R^d=\R^{d_1}\oplus\ldots\oplus\R^{d_r}$,
denote by $\sigma_i(x)$ the sum of the coordinates of $x\in\R^d$ that correspond
to the subspace $\R^{d_i}$, i.e., $\sigma_i(x)=\sum_{k=1}^{d_i} x_k^i$.
With each integer point $x\in\R^d$ in the string space, we associate the
{\em weight} $p(x)\in\R^r$ defined as $(\sigma_1(x),\ldots,\sigma_r(x))$.
For the rest of the paper, we will always assume that $l_i(x)$ depends only on $p(x)$, that is, $l_i$ comes from a linear function on $\R^r$ (the latter will also be denoted by $l_i$).
In addition, we assume that $l_i$ is integral, i.e., $l_i(x)\in\Z$ for all $x\in\Z^n$.

Denote the basis vectors in $\R^r$ by $\a_1$, \ldots, $\a_r$, and denote the coordinates with respect to this basis by $(y_1,\ldots,y_r)$.
For each $i=1$, \ldots $r$, define the affine reflection
$s_i:\R^r\to\R^r$ by the formula
$$s_i(y_1,\ldots,y_i,\ldots,y_r)=(y_1,\ldots,l_i(y)-y_i,
\ldots,y_r).$$

\begin{example} \label{e.GL}
For the string space $\mathbb R^d=\mathbb R^{n-1}\oplus\mathbb R^{n-2}\oplus\ldots\oplus\mathbb R^1$ from Example \ref{e.GZ}, define the functions $l_i$ by the formula
$$l_i(x)=\sigma_{i-1}(x)+\sigma_{i+1}(x),$$
where we put $\sigma_0=\sigma_n=0$.
Identify $\R^{n-1}$ with the weight lattice of $SL_n$ so that
$\a_i$ is identified with the $i$-th simple root.
In this case, the reflection $s_i$ coincides with the simple reflection in the hyperplane perpendicular to the root $\a_i$.
\end{example}

We now consider the ring $R$ of Laurent polynomials in the formal exponentials
$t_1:=e^{\a_1}$, \ldots, $t_n:=e^{\a_n}$ (that is, $R$ is the group algebra of the lattice $\Z^n\subset\R^n$).
Let $P\subset\R^d$ be a {\em lattice} polytope in the string space, i.e., the vertices of $P$ belong to $\Z^n$.
Define the {\em character of $P$} as the  sum of formal exponentials $e^{p(x)}$ over all
integer points $x$ inside and at the boundary of $P$:
$$\chi(P):=\sum_{x\in P\cap\Z^d}e^{p(x)}.$$
In particular, if $d=r$, then $\chi(P)$ is exactly the integer point transform of $P$.
The $R$-valued function $\chi$ can be extended by linearity to all {\em lattice} convex chains, that is, to the chains $\sum_P c_P {\I}_P$ such that $P$ is a lattice polytope and $c_P\in\Z$.

Define {\em Demazure operator} $T_i$ on $R$ as follows:
$$[T_if](y)=\frac{f(y)-t_is_if(y)}{1-t_i},$$
where $s_if(y):=f(s_iy)$.
For the string space of Example \ref{e.GL}, these operators reduce to the
classical Demazure operators on the group algebra of the weight lattice of $SL_n$.

The following result motivates Definition \ref{d.main}.
\begin{thm} \label{t.Demazure} Let $P\subset \R^d$ be a lattice parapolytope.
Then
$$\chi(D_i(\I_P))=T_i\chi(P).$$
\end{thm}
\begin{proof}
By definition of $D_i(\I_P)$, it suffices to prove this identity when $P=c+\Gamma$, where $c$ lies in the complement to $\R^{d_i}$ and  $\Gamma:=\Pi(\mu,\nu)\subset\R^{d_i}$ is a coordinate parallelepiped.
Then
$$\chi(P)=e^{p(c)}\sum_{z\in \Gamma\cap\Z^{d_i}}t_i^{\sigma(z)}.$$
Hence,
$$T_i(\chi(P))=e^{p(c)}T_i\left(\sum_{z\in \Gamma\cap\Z^{d_i}}t_i^{\sigma(z)}\right).$$
Recall that by definition of $\nu'$ we have
$$\sum_{k=1}^{d_i}(\mu_k+\nu'_k)=l_i(c).$$
Assume that $\nu_j'\ge\nu_j$.
Let $\Pi$ denote $\Pi(\mu,\nu ')$.
Then $\Gamma$, $\Pi$ and $T_i$ satisfy the hypothesis of \cite[Proposition 6.3]{KST}.
Applying this proposition we get that
$$T_i\left(\sum_{z\in \Gamma\cap\Z^{d_i}}t_i^{\sigma(z)}\right)=\sum_{z\in\Pi\cap\Z^{d_i}}t_i^{\sigma(z)}.$$
Hence, $T_i(\chi(P))=\chi(D_i(P))$.

The case $\nu'_j<\nu_j$ is completely analogous.
\end{proof}
Note that Theorem \ref{t.Demazure} for $d_i=1$ follows directly from the definitions of $T_i$ and $D_i$ (see Remark \ref{r.segment}).

\begin{example} \label{e.trap_square} Figure 3 illustrates Theorem \ref{t.Demazure} when $d_i=2$ and $P=c+\Gamma$ where $\Gamma\subset \R^{d_i}$ is the segment $[(-1,-1),(2,-1)]$.
Namely, $T_i(t_i^{x_1^i+x_2^i})$ is equal to the character of the segment
$[(x_1^i,x_2^i),(x_1^i,l_i(c)-2x_1^i-x_2^i)]$ for every $(x_1^i,x_2^i)\in \Gamma\cap\Z^2$
by definition of $T_i$.
Hence, $$\sum_{(x_1^i,x_2^i)\in \Gamma\cap\Z^2}T_i(t_i^{x_1^i+x_2^i})$$ for $l_i(c)=3$ coincides with the character of the trapezoid shown on Figure 3 (left).
It is easy to construct a bijective correspondence between the integer points in the trapezoid and those in the rectangle $D_i(P)$ in such a way that the sum of coordinates is preserved.
The former are marked by black dots, and the latter by empty circles.
\begin{figure}\label{f.trap_square}
\includegraphics[width=5cm]{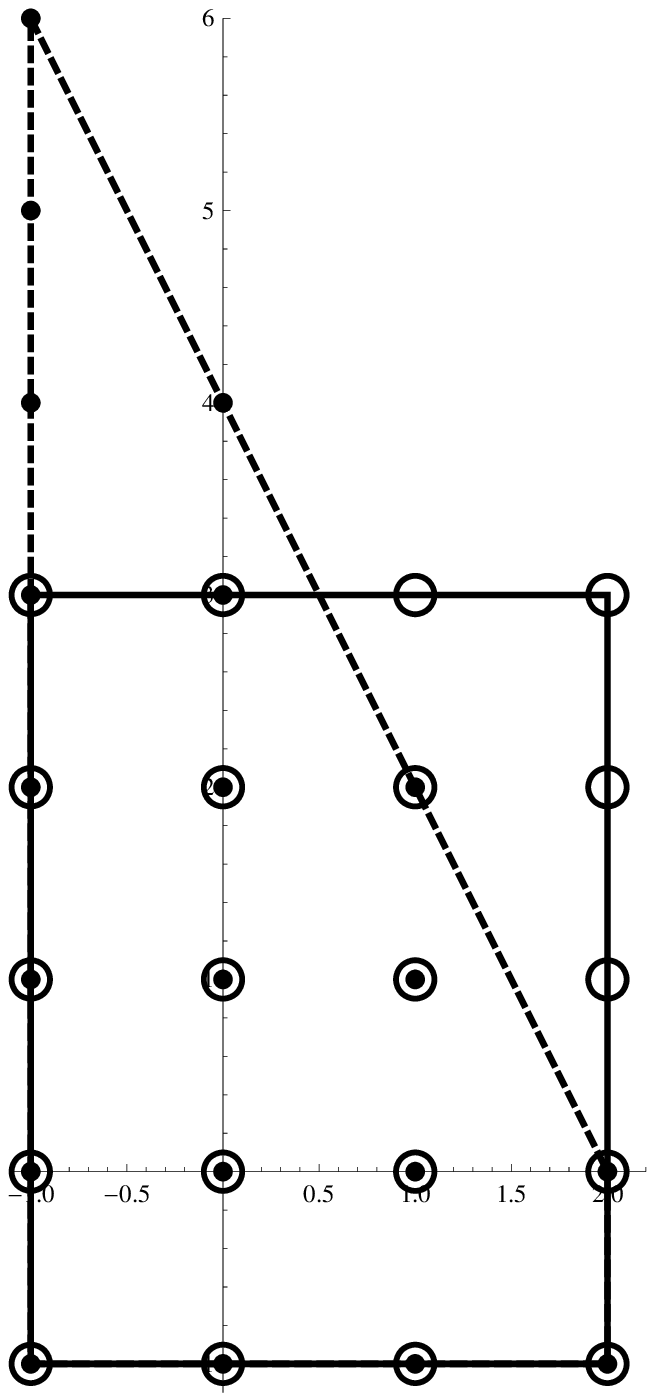} \quad
\includegraphics[width=5cm]{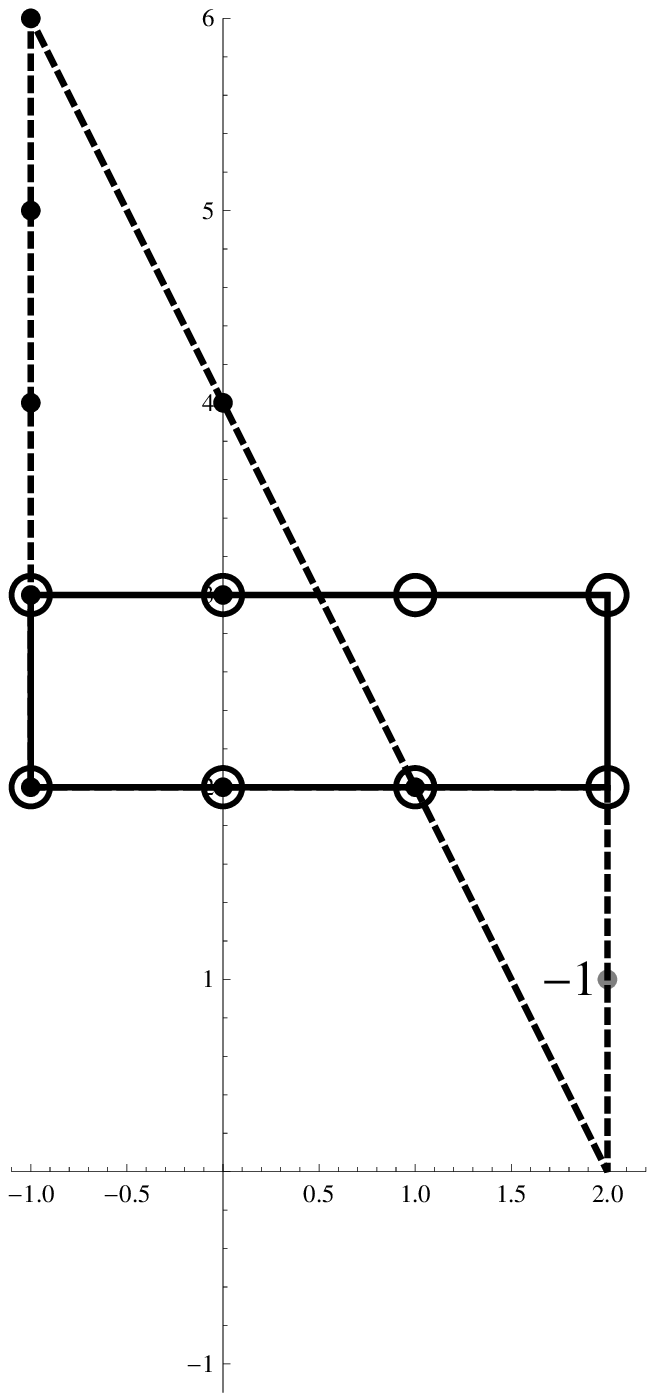}
\caption{Rectangle and trapezoid yield the same character}
\end{figure}
\end{example}

Theorem \ref{t.Demazure} allows one to construct various polytopes (possibly virtual) and convex chains whose characters yield the Demazure characters (in particular, the Weyl character) of irreducible representations of reductive groups (see Section \ref{ss.reductive}).
The same character can be captured using string spaces for different partitions $d=d_1+d_2+\ldots+d_n$ (see Section \ref{ss.degen}).
The case $d_1=\ldots=d_n=1$ produces polytopes with very simple combinatorics, namely, multidimensional versions of trapezoids that are combinatorially equivalent to cubes (they are called {\em twisted cubes} in \cite{GK}).
However, twisted cubes that represent the Weyl characters are virtual.
Considering string spaces with $d_i>1$ allows one to represent the Weyl character by a true though more intricate polytope (see Example \ref{e.GK} for $SL_3$ and Section \ref{ss.degen}).
The reason is illustrated by Figure 3 (right) that depicts a virtual trapezoid and a (true) rectangle with the same character.
Note that the point $(2,1)$ (marked by $-1$) contributes a negative summand to the character of the trapezoid.

\subsection{Gelfand--Zetlin polytopes for $SL_n$}\label{ss.GL}
Let $\mathbb R^d=\mathbb R^{n-1}\oplus\mathbb R^{n-2}\oplus\ldots\oplus\mathbb R^1$ be the string space of rank $(n-1)$ from Example \ref{e.GL}.
The theorem below shows how to construct the classical Gelfand--Zetlin polytopes (see Example \ref{e.GZ}) via the convex-geometric Demazure operators $D_1$,\ldots, $D_{n-1}$.
\begin{thm}\label{t.GZ} For every strictly dominant weight $\l=(\l_1,\ldots,\l_n)$ (that is, $\l_1>\ldots>\l_n$) of $GL_n$ such that $\l_1+\ldots+\l_n=0$, the Gelfand--Zetlin polytope $Q_\l$ coincides with the polytope
$$\left[(D_1)(D_2D_1)(D_3D_2D_1)\ldots(D_{n-1} \ldots D_1)\right](a_\l),$$
where $a_\l\in\R^d$ is the point $(\l_2,\ldots,\l_n;\l_3,\ldots,\l_n;
\ldots;\l_n)$.
\begin{proof}
Let us define the polytope
$$P_\l(i,j):=\left[(\hat D_{n-j}\ldots \hat D_i D_{i-1}\ldots D_1)\ldots(D_{n-1} \ldots D_1)\right](a_\l)$$
for every pair $(i,j)$ such that $1\le i\le (n-j)\le (n-1)$.
Put $x^0_l=\l_i$ for $l=1$,\ldots, $n$.
We will show by induction on dimension that $P_\l(i,j)$ is the face of the Gelfand--Zetlin polytope $Q_\l$ given by the equations $x^k_l=x^{k-1}_{l+1}$ for all pairs $(k,l)$ such that either $l>j$, or $l=j$ and $k\ge i$.
The induction base is $P_\l(1,1)=a_\l$, which is clearly a vertex of $Q_\l$ by our assumption.
The induction step follows from Lemma \ref{l.GZ} below.
Hence, $P_\l(1,n-1)$ is the facet of $Q_\l$ given by the equation $x^1_{n-1}=\l_{n}$.
Applying Lemma \ref{l.GZ} again, we get that $D_1(P_\l(1,n-1))=Q_\l$.
\end{proof}
\end{thm}

Note that any Gelfand--Zetlin polytope $Q_\l$ can be obtained by a parallel translation from one with $\l_1+\ldots+\l_n=0$.

The lemma below can be easily deduced directly from the definition of $D_i$ using Example \ref{e.GZ} together with an evident observation that $a+b=\min\{a,b\}+\max\{a,b\}$ for any $a,b\in\R$.
\begin{lemma} \label{l.GZ} Let $\Gamma$ be a face of the Gelfand--Zetlin polytope $Q_\l$ given by the following equations
$$\begin{array}{lllllllll}
x^{i-1}_1&x^{i-1}_2 &x^{i-1}_3 &\ldots&x^{i-1}_{j}&x^{i-1}_{j+1}& x^{i-1}_{j+2}&\ldots&   x^{i-1}_{n-i+1}\\
         &          &          &      && \|      &\|            &      &   \|             \\
         &x^i_1     &x^i_2     &\ldots&x^{i}_{j-1}&x^i_{j}& x^i_{j+1}      &\ldots&   x^i_{n-i}      \\
         &          &          &      &&         &\|            &      &   \|              \\
         &          &x^{i+1}_1 &\ldots&x^{i+1}_{j-2}&x^{i+1}_{j-1}& x^{i+1}_{j}&\ldots&x^{i+1}_{n-i-1}\\
\end{array}.
$$
as well as by (possibly) other equations that do not involve variables $x_1^i$,\ldots $x_{n-i}^i$.
Then the defining equations of $D_i(\Gamma)$ are obtained from those of $\Gamma$ by removing the equation $x^i_j=x^{i-1}_{j+1}$.
\end{lemma}

Recall that integer points inside and at the boundary of the Gelfand--Zetlin polytope $Q_\l$ by definition of this polytope parameterize a natural basis ({\em Gelfand--Zetlin basis}) in the irreducible representation of $GL_n$ with the highest weight $\l$.
Under this correspondence, the map $p:\R^{d}\to\R^{n-1}$ assigns to every integer point the weight of the corresponding basis vector.
Combining Theorem \ref{t.GZ} with Theorem \ref{t.Demazure} one gets a combinatorial proof of the Demazure character formula for the decomposition $w_0=(s_1)(s_2s_1)(s_3s_2s_1)\ldots (s_{n-1}s_{n-2}\ldots s_1)$ of the longest word in $S_n$ (in this case, the Demazure character coincides with the Weyl character of $V_\l$).
Here $s_i$ denotes the elementary transposition $(i,i+1)\in S_n$.

\subsection{Applications to arbitrary reductive groups}\label{ss.reductive} We now generalize Gelfand--Zetlin polytopes to other reductive groups using Theorem \ref{t.Demazure}.
Let $G$ be a connected reductive group of semisimple rank $r$.
Let $\alpha_1$,\ldots, $\alpha_r$ denote simple roots of $G$, and
$s_1$,\ldots, $s_r$ the corresponding simple reflections.
Fix a reduced decomposition $w_0=s_{i_1}s_{i_2}\cdots s_{i_d}$ where $w_0$ is the longest element of the Weyl group of $G$.
Let $d_i$ be the number of $s_{i_j}$ in this decomposition such that $i_j=i$.
Consider the string space
$$\mathbb R^d=\mathbb R^{d_1}\oplus\ldots\oplus\mathbb R^{d_r},$$
where the functions $l_i$ are given by the formula:
$$l_i(x)=\sum_{k\ne i} (\alpha_k,\alpha_i)\sigma_k(x).$$
Here $(\alpha_k,\alpha_i)$ is determined by the simple reflection $s_i$ as follows:
$$s_i(\alpha_k)=\alpha_k+(\alpha_k,\alpha_i)\alpha_i,$$
(that is, the function $(\cdot,\a_i)$ is minus the coroot corresponding to $\a_i$).
In particular, if $G=SL_n$ and $w_0=(s_1)(s_2s_1)(s_3s_2s_1)\ldots(s_{n-1}\ldots s_1)$, then we get the string space from Example \ref{e.GL}.

Define the projection $p$ of the string space to the real span $\R^r$ of the weight lattice of $G$ by the formula $p(x)=\sigma_1(x)\alpha_1+\ldots+\sigma_r(x)\alpha_r$.

\begin{thm} \label{t.Weyl} For every dominant weight $\lambda$ in the root lattice of $G$, and every point $a_\lambda\in\mathbb \Z^d$ such that $p(a_\l)=w_0\l$ the convex chain
$$P_\l:=D_{i_1}D_{i_2}\ldots D_{i_d}(a_\lambda)$$
yields the Weyl character $\chi(V_\lambda)$ of the irreducible $G$-module $V_\lambda$, that is,
$$\chi(V_\lambda)=\chi(P_\l).$$
\end{thm}
\begin{proof} By the Demazure character formula \cite{Andersen} we have
$$\chi(V_\lambda)=T_{i_1}\ldots T_{i_d}e^{w_0\l}.$$
This formula together with Theorem \ref{t.Demazure} implies by induction the desired statement.
\end{proof}

As a corollary, we get that $p_*(P_\l)$ is the weight polytope of $V_\l$ in $\R^r$.
Here $p_*$ denotes the push-forward of convex chains (see \cite[Proposition-Definition 2]{PKh}).

\begin{remark} A slight modification of Theorem \ref{t.Demazure} makes it applicable to all dominant weights (not only those inside the root lattice).
Namely, instead of the lattices $\Z^d\subset\R^d$ and $\Z^r\subset\R^r$ one should consider the shifted lattices $a_\l+\Z^d\subset\R^d$ and $\l+\Z^r\subset\R^r$, and define characters of polytopes with respect to these new lattices.
The convex chain $P_\l$ will be lattice with respect to the lattice $a_\l+\Z^d$.
\end{remark}

In the same way, we can construct convex chains that capture the characters of Demazure submodules of $V_\l$ for any element $w$ in the Weyl group and a reduced decomposition $w=s_{j_1}\ldots s_{j_\ell}$ (see Corollary \ref{c.Demazure}). In particular, if $s_{j_1}\ldots s_{j_\ell}$ is a terminal subword of $s_{i_1}s_{i_2}\cdots s_{i_d}$ (that is, $j_\ell=i_d$, $j_{\ell-1}=i_{d-1}$, etc.) then the corresponding convex chain is a face of $P_\l$.
It is interesting to check whether this convex chain is always a true polytope.
One way to do this would be to identify it with an Okounkov polytope of the Bott--Samelson resolution corresponding to the word $s_{j_1}\ldots s_{j_\ell}$ (see Conjecture \ref{conj}).

\subsection{Examples}
\paragraph{ $\bf Sp(4)$.} \label{e.Sp}Take $G=Sp(4)$ and $w_0=s_2s_1s_2s_1$ (here $\a_1$ denotes the shorter root and $\a_2$ denotes the longer one).
The corresponding string space of rank $2$ is $\R^4=\R^2\oplus\R^2$ together with $l_1=2(x_1^2+x_2^2)$ and $l_2=x_1^1+x_2^1$.
Let $\l=-p_1\a_1-p_2\a_2$ be a dominant weight, that is, $\l_1:=(p_2-p_1)\ge 0$ and $\l_2:=(p_1-2p_2)\ge 0$.
Choose a point $a_\l=(a,b,c,d)$ such that $(a+b)=p_1$ and $(c+d)=p_2$ (that is, $p(a_\l)=w_0\l=-\l$).
Label coordinates in $\R^4$ by $x:=x_1^1$, $y:=x_2^1$, $z:=x_1^2$ and $t:=x_2^2$.
Then the polytope $D_2D_1D_2D_1(a_\l)$ is given by inequalities
$$0\le x-a\le 2\l_1, \quad z-c\le x-a+\l_2, \quad y-b\le 2(z-c),$$
$$y-b\le z-c+\l_2,\quad 0\le t-d\le \l_2, \quad t-d\le \frac{y-b}{2}.$$
It is not hard to show that the polytopes $D_1D_2D_1D_2(a_\l)$ and $D_2D_1D_2D_1(a_\l)$ are the same up to a linear transformation of $\R^4$.
Each polytope has 11 vertices, hence, they are not combinatorially equivalent to string polytopes for $s_1s_2s_1s_2$ or $s_2s_1s_2s_1$ defined in \cite{L}.

\paragraph{$\bf SL(3)$.} \label{e.GK} Take $G=SL(3)$ and $w_0=s_1s_2s_1$.
The corresponding string space of rank $2$ coincides with the one from Section \ref{ss.examples}, namely, $\R^3=\R^2\oplus\R$, and  $l_1=x_1^2$, $l_2=x_1^1+x_2^1$.
If $a_\l=(b,c,c)$ where $-b-c\ge b\ge c$, then the polytope $D_1D_2D_1(a_\l)$ is the Gelfand--Zetlin polytope $Q_\l$ for $\l=(-b-c,b,c)$.
Label coordinates in $\R^3$ by $x:=x_1^1$, $y:=x_2^1$ and $z:=x_1^2$.
We now introduce a different structure of a string space in $\R^3$ by splitting $\R^2$, namely, $\R^3=\R^1\oplus\R^1\oplus\R^1$ with coordinates $\tilde x_1^1$, $\tilde x_1^2$, $\tilde x_1^3$
such that $\tilde x_1^1=x$, $\tilde x_1^2=z$, $\tilde x_1^3=y$.
Put $\tilde l_1=l_1-2y$, $\tilde l_2=l_2$ and $\tilde l_3=l_1-2x$.
It is easy to check that $\tilde D_2\tilde D_1(a_\l)=D_2D_1(a_\l)$ and deduce by arguments of Example \ref{e.trap_square} that the virtual polytope $\tilde D_3\tilde D_2\tilde D_1(a_\l)$ (see Figure 4) has the same character as the polytope $D_1D_2D_1(a_\l)$.
In particular, the image of $\tilde D_3\tilde D_2\tilde D_1(a_\l)$ under the projection $(x,y,z)\mapsto (x+y,z)$ coincides with the weight polytope of the irreducible representation of $SL_3$ with the highest weight $-c\a_1-(a+b)\a_2$ (provided that the latter is dominant, that is, $a+b-2c\ge0$ and $c-2(a+b)\ge0$).
The virtual polytope $\tilde D_3\tilde D_2\tilde D_1(a_\l)$ is a twisted cube of Grossberg--Karshon (cf. \cite[Figure 2]{GK}) given by the inequalities
$$a\le x\le c-2b-a, \quad c\le z\le x+b-c,
\quad b\le y\le -2x+z-b. \eqno{(GK)}
$$
Note that the last pair of inequalities is inconsistent when $b>-2x+z-b$, and should be interpreted in the sense of convex chains.
More precisely, $\tilde D_3\tilde D_2\tilde D_1(a_\l)$=$\I_{P}-\I_{Q}$, where $P$ is the convex polytope given by inequalities $(GK)$
and $Q$ is the set given by the inequalities
$$a\le x\le c-2b-a, \quad c\le z\le x+b-c,
\quad b> y> -2x+z-b. $$
(cf. \cite[Formula (2.21)]{GK}).
\begin{figure}\label{f.GK}
\includegraphics[width=10cm]{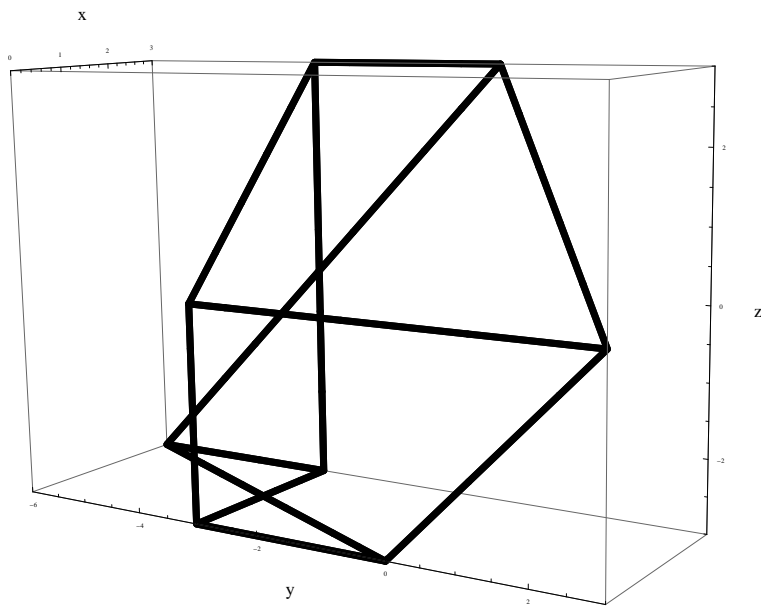}
\caption{Virtual polytope $\tilde D_3\tilde D_2\tilde D_1(a_\l)$ for $a_\l=(0,-3,-3)$}
\end{figure}
A generalization of this example will be given in Section \ref{ss.degen}.

\section{Bott towers and Bott--Samelson resolutions} \label{s.ag}
In this section, we outline possible algebro-geometric applications of the convex-geometric Demazure operators.
\subsection{Bott towers}\label{ss.Bott}
Let us recall the definition of a {\em Bott tower} (see \cite{GK} for more details).
It is a toric variety obtained from a point by iterating the following step.
Let $X$ be a toric variety, and $L$ a line bundle on $X$.
Define a new toric variety $Y:=\P(L\oplus \Oc_X)$ as the projectivization of the split rank
two vector bundle $L\oplus \Oc_X$ on $X$.
Consider a sequence of toric varieties
$$Y_0\leftarrow Y_1\leftarrow\ldots\leftarrow Y_d,$$
where $Y_0$ is a point, and $Y_i=\P(L_{i-1}\oplus\Oc_{Y_{i-1}})$ for a line bundle $L_{i-1}$
on $Y_{i-1}$.
In particular, $Y_1=\P^1$ and $Y_2=\P(\Oc_{\P^1}\oplus\Oc_{\P^1}(k))$ is a Hirzebruch surface.
We call $Y_d$ the {\em Bott tower} corresponding to the collection of line bundles
$(L_1,\ldots,L_{d-1})$.
Note that the collection $(L_1,\ldots,L_{d-1})$ depends on $\frac{d(d-1)}{2}$ integer parameters since
$\Pic(Y_i)=\Z^i$.
Recall that the Picard group of a toric variety of dimension $d$ can be identified with a group of virtual lattice polytopes in $\R^d$ in such a way that very ample line bundles get identified with their Newton polytopes.
One can describe the (possibly virtual) polytope $P(\mathcal L)$ of a given line bundle $\mathcal L$  on $Y_d$ using a suitable string space.

Consider a string space with $d=r$, that is, $d_1=\ldots=d_r=1$.
We have the decomposition
$${\mathbb R}^d=\underbrace{{\mathbb R}\oplus\ldots\oplus{\mathbb R}}_{d}.$$
Label coordinates in $\R^d$ as follows: $x^i_1:=y_i$ for $i=1$,\ldots, $d$.
Since we will be interested in the polytope $P:=D_1\ldots D_d(a)$, we can assume that the linear function $l_i$ for $i<d$ does not depend on
$y_1$, \ldots, $y_i$, and $l_d=y_1$.
Hence, the collection $(l_1,\ldots,l_{d-1})$ of linear functions
also depends on $\frac{d(d-1)}{2}$ parameters.

The projective bundle formula gives a natural basis ($\eta_1$, \ldots, $\eta_d$) in the Picard group of $Y_d$.
Namely, for $d=1$, the basis in $\Pic(\P^1)$ consists of the class of a point in $\P^1$.
We now proceed by induction.
Let ($\eta_1$, \ldots, $\eta_{i-1}$) be the basis in $\Pic(Y_{i-1})$ (we identify $\Pic(Y_{i-1})$ with its pull-back to $\Pic(Y_d)$).
Put $\eta_i=c_1(\Oc_{Y_i}(1))$ where $c_1(\Oc_{Y_i}(1))$ denotes the first Chern class of the tautological quotient line bundle $\Oc_{Y_i}(1)$ on $Y_i$.
Decompose $L_1$,\ldots, $L_{d-1}$ in the basis ($\eta_1$, \ldots, $\eta_d$):
$$L_1=a_{1,1}\eta_1, \quad
\ldots, \quad L_{d-1}=a_{d-1,1}\eta_1+\ldots+a_{d-1,d-1}\eta_{d-1}. \eqno(*)$$
Similarly, decompose $l_1$,\ldots, $l_{d-1}$ in the basis of coordinate functions $(y_1,\ldots,y_d)$:
$$l_1=b_{1,1}y_2+\ldots+b_{1,d-1}y_d, \quad \ldots, \quad
l_{d-1}=b_{d-1,d-1}y_d.
\eqno(**).$$

Let $Y_d$ be the Bott tower corresponding to the collection $(*)$ of line bundles, and $\R^d$ the string space corresponding to the collection $(**)$ of linear functions.
One can show (cf. \cite[Theorem 3]{GK}) that if $a_{i,j}=b_{j,i}$, then there exists $a_\mathcal L\in\R^d$ such that
$$P(\mathcal L)=D_1D_2\ldots D_d(a_\mathcal L).$$

In particular, when $\mathcal L$ is very ample the polytope $D_1D_2\ldots D_d(a_\mathcal L)$ coincides with the Newton polytope of the pair $(Y_d,\mathcal L)$.
Note that the intermediate polytopes $\{a_\mathcal L\}\subset D_d(a_\mathcal L)\subset D_{d-1}D_d(a_\mathcal L)\subset\ldots$ 
correspond to the flag of toric subvarieties $Z_0=\{pt\}\subset Z_1\subset\ldots\subset Z_d=Y_d$, where $Z_i=p_{d-i}^{-1}(Z_0)$ and $p_i$ is the projection $p_i:Y_d\to Y_i$.

\subsection{Bott--Samelson varieties.} \label{ss.BS}
\begin{figure}\label{f.BS}
\includegraphics[width=10cm]{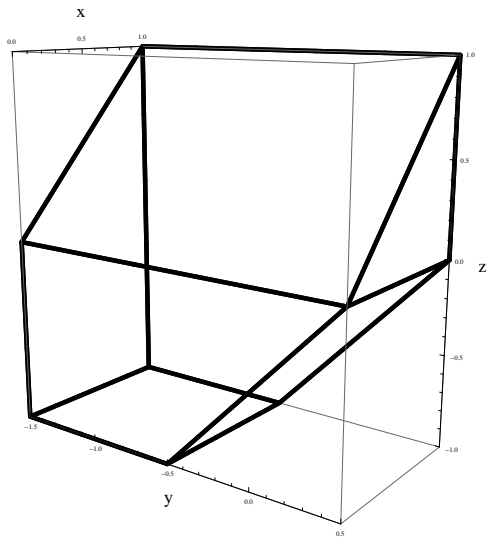}
\caption{Polytope $D_1 E_u D_2 D_1(a)$ for $a=(0,-1,-1)$ and $u=(0,-1/2,0)$}
\end{figure}

Similarly to Bott towers, Bott--Samelson varieties can be obtained by successive projectivizations of rank two vector bundles.
In general, these bundles are no longer split, so the resulting varieties are not toric.
In \cite{GK}, Bott--Samelson varieties were degenerated to Bott towers by changing complex structure (in particular, Bott--Samelson varieties are diffeomorphic to Bott towers when regarded as real manifolds).
Below we define these varieties using notation of Section \ref{ss.reductive}.

Fix a Borel subgroup $B\subset G$.
With every collection of simple roots $(\alpha_{i_1},\ldots,\alpha_{i_\ell})$, one can associate
a {\em Bott--Samelson variety} $R_{(i_1,\ldots,i_\ell)}$ and a map $R_{(i_1,\ldots,i_\ell)}\to G/B$ by the following inductive procedure.
Put $R_{\emptyset}=pt$.
For every $\ell$-tuple $I=(i_1,\ldots,i_\ell)$ denote by $I^j$ the $(\ell-1)$-tuple $(i_1,\ldots,\hat i_j,\ldots,i_\ell)$.
Define $R_{I}$ as  the fiber product $R_{I^\ell}\times_{G/P_{\ell}}G/B $, where $P_{i_\ell}$ is the minimal parabolic subgroup corresponding to the root $\a_{i_\ell}$.
The map  $r_{I}:R_{I}\to G/B$ is defined as the projection to
the second factor.
There is a natural embedding
$$R_{I^\ell}\hookrightarrow R_I;\quad x\mapsto (x,r_I(x)).$$
In particular, any subsequence $J\subset I$ yields
the embedding $R_J\hookrightarrow R_I$.

It follows from the projective bundle formula that the Picard group of $R_I$ is freely generated by the divisors $R_{I^1}$,\ldots, $R_{I^\ell}$.
Denote by $v$ the geometric valuation on $\C(R_I)$ defined by the flag $R_{\emptyset}\subset R_{(i_\ell)}\subset R_{(i_{\ell-1},i_\ell)}\subset\ldots\subset R_{I^1}\subset R_I$.
Let $\mathcal L$ be a line bundle on $R_I$, and $P_v(\mathcal L)$ its Okounkov body with respect to the valuation $v$.
Conjecturally, $P_v(\mathcal L)$ can be described using string spaces as follows.

Replace a reduced decomposition of $w_0$ in the definition of the string space from Section \ref{ss.reductive} by a sequence $(\alpha_{i_1},\ldots,\alpha_{i_\ell})$ that defines $R_I$ (we no longer require that $s_{i_1}\ldots s_{i_\ell}$ be reduced).
More precisely, let $d_i$ the number of $\a_{i_j}$ in this sequence such that $i_j=i$.
We get the following string space $\mathcal S_I$ of rank $\le r$  and dimension $\ell$:
$$\mathbb R^\ell=\mathbb R^{d_1}\oplus\ldots\oplus\mathbb R^{d_r},$$
where the functions $l_i$ are given by the formula:
$$l_i(x)=\sum_{k\ne i} (\alpha_k,\alpha_i)\sigma_k(x). \eqno(BS)$$
In particular, if $\ell=d$ and $s_{i_1}\cdots s_{i_\ell}$ is reduced then $R_I$ is a Bott--Samelson resolution of the flag variety $G/B$, and $\R^\ell$ is exactly the string space from Section \ref{ss.reductive}.
Denote by $E_u$ the parallel translation in the string space by a vector
$u\in{\mathbb R}^\ell$.
\begin{conj} \label{conj}
For every line bundle $\mathcal L$ on $R_I$, there exists a point $\mu\in\R^r$ and vectors
$u_1,\ldots,u_{\ell}\in\R^\ell$ such that we have
$$P(\mathcal L)=E_{u_1}D_{i_1}E_{u_2}D_{i_2}\ldots E_{u_\ell}D_{i_\ell}(a_\mu)$$
for any point $a_\mu\in\R^\ell$ that satisfies $p(a)=\mu$.
\end{conj}
In particular, if $\mathcal L=r_I^*\mathcal L(\l)$, where $\mathcal L(\l)$ is the line bundle on $G/B$ corresponding to the dominant weight $\l$, then one can take $u_2=\ldots=u_\ell=0$ and $\mu=\lambda$.
This conjecture agrees with the example computed in \cite[Section 6.4]{Anderson} for $SL_3$ and the Bott--Samelson resolution $R_{(1,2,1)}$ (cf. Figure 5 and Figure 3(b) in loc.cit.).
Figure 5 shows the polytope $D_1 E_u D_2 D_1(a)$ for the string space $(BS)$ when $G=SL_3$ and $I=(1,2,1)$.

\subsection{Degenerations of string spaces} \label{ss.degen}
While twisted cubes of Grossberg--Karshon for $GL_n$ and Gelfand--Zetlin polytopes have different combinatorics they produce the same Demazure characters.
We now reproduce this phenomenon for general string spaces.
In particular, we transform the string space $(BS)$ from Section \ref{ss.BS}  into a string space from Section \ref{ss.Bott}.

Let $\mathcal S$ be a  string space $\R^d=\R^{d_1}\oplus\ldots\oplus\R^{d_r}$ with functions $l_1$,\ldots, $l_r$.
Suppose that $d_i>1$.

\begin{defin} The {\em $i$-th degeneration} of the string space $\mathcal S$ is the string space
$$\R^d=\R^{d_1}\oplus\ldots\oplus\underbrace{\R^{d_i-1}\oplus\R^{1}}_{\R^{d_i}}\oplus\ldots\oplus\R^{d_r}$$
of rank $(r+1)$ with functions $l_1$,\ldots, $l_i'$, $l_i''$,\ldots, $l_r$, where  $$l_i'(x)=l_i(x)-2x^i_{d_i}; \quad l_i''(x)=l_i(x)-2\sum_{k=1}^{d_i-1} x_k^i.$$
\end{defin}

\begin{example} The string space $\R\oplus\R\oplus\R$ from Example \ref{e.GK} is the 1-st degeneration of the space $\R^2\oplus\R$ with $l_1=x_1^1+x_2^1$, $l_2=x_1^2$.
\end{example}

Define the projection $p_i:\R^{r+1}\to \R^r$ by sending $(y_1,\ldots,y_i',y_i'',\ldots,y_r)$ to $(y_1,\ldots,y_i'+y_i'',\ldots,y_r)$.
This projection induces a homomorphism of group algebras of the lattices $\Z^{r+1}$ and $\Z^r$, which we will also denote by $p_i$.
It is easy to check that
$$T_i\circ p_i=p_i\circ T_i'=p_i\circ T_i''.$$
Combining this observation with Theorem \ref{t.Demazure}, we get the following proposition.
\begin{prop} \label{p.degen}
For any polytope $P\subset \R^d$ such that $\dim P\cap\R^{d_i}<d_i$, we have
$$\chi(D_i(P))=p_i(\chi(D_i'(P)))=p_i(\chi(D_i''(P))).$$
\end{prop}

We now degenerate successively  the string spaces from Section \ref{ss.BS}.
Let $I=(\alpha_{i_1},\ldots,\alpha_{i_\ell})$ be a sequence of simple roots, and
$$\mathbb R^\ell=\mathbb R^{d_1}\oplus\ldots\oplus\mathbb R^{d_r}$$
is the corresponding string space $\mathcal S_I$ with the functions $l_1$,\ldots, $l_r$ given by $(BS)$.

Let $\tilde{\mathcal S_I}$ be the string space of rank $\ell$ obtained from $S_I$ by $(d_1-1)$ first degenerations, $(d_2-1)$ second degenerations etc., that is,
$$\tilde{\mathcal S_I}=\underbrace{\R^{(1)}_{1}\oplus\ldots\oplus\R^{(1)}_{d_1}}_{\R^{d_1}}\oplus\ldots\oplus
\underbrace{\R^{(\ell)}_{1}\oplus\ldots\oplus\R^{(\ell)}_{d_\ell}}_{\R^{d_\ell}},$$
where the functions $l^{(i)}_j$ are given by the formula:
$$l^{(i)}_j(x)=l_i(x)-2\sum_{k\ne j} x^i_k.$$
Denote the Demazure operators associated with $\tilde{\mathcal S_I}$ by $\tilde D_{j}^{(i)}$.

For a point $a\in\R^\ell$, consider the convex chain $P_{I}=D_{i_1}\ldots D_{i_\ell}(a)$.
For every $k=1$,\ldots, $r$, we now replace the rightmost $D_k$ in the expression $D_{i_1}\ldots D_{i_\ell}$ by $D_{1}^{(k)}$, the next one by $D^{(k)}_{2}$,\ldots, the leftmost by $D^{(k)}_{i_k}$.
Denote the resulting convex chain by $\tilde P_I$.

\begin{example} If $r=3$, $\ell=6$ and $I=(\a_1,\a_2,\a_1,\a_3,\a_2,\a_1)$, then
$$P_I=D_1D_2D_1D_3D_2D_1, \quad \tilde P_I=\tilde D_{3}^{(1)}\tilde D_{2}^{(2)}\tilde D_{2}^{(1)}\tilde D_{1}^{(3)}\tilde D_{1}^{(2)}\tilde D_{1}^{(1)}(a).$$
\end{example}

Proposition \ref{p.degen} implies that $P_I$ and $\tilde P_I$
have the same character (with respect to the map $p:x\mapsto\sigma_1(x)\alpha_1+\ldots+\sigma_r(x)\alpha_r$).

\begin{cor}\label{c.Demazure}
If $p(a)$ is dominant and $s_{i_1}\cdots s_{i_\ell}$ is reduced, then the corresponding Demazure character coincides with $\chi(P_I)=\chi(\tilde P_I)$.
\end{cor}
The proof is completely analogous to the proof of Theorem \ref{t.Weyl}.

\begin{remark}\label{r.GK}
Note that the convex chain $\tilde P_I$ coincides with the twisted cube constructed in \cite{GK} for the corresponding Bott--Samelson resolution.
Indeed, $(\a_i,\a_i)=-2$ according to our definition of the function $(\cdot,\a_i)$ (see Section \ref{ss.reductive}), hence, we can write
$$l^{(i)}_j(x)=\sum_{(p,q)\ne (i,j)}(\alpha_p,\alpha_i)x^p_q.$$
It is now easy to check that the defining inequalities for $\tilde P_I$ coincide with the inequalities given by \cite[Formula (2.21)]{GK} together with computations of \cite[Section 3.7]{GK}.
\end{remark}

The string space $\mathcal S_I$ and its complete degeneration $\tilde{\mathcal S_I}$ are two extreme cases that yield convex chains for given Demazure characters.
By taking partial degenerations of $\mathcal S_I$ one can construct intermediate convex chains with the same character.
However, only $\mathcal S_I$ might produce true polytopes (such as Gelfand--Zetlin polytopes for $G=SL_n$ or polytope of Example \ref{e.Sp} for $G=Sp_4$) that represent the Weyl characters.
Indeed, such a polytope must have a face $D_i^2(a)$ (in the case of $\mathcal S_I$) or a face $D_i'D_i''(a)$ (in the case of the $i$-th degeneration of $\mathcal S_I$).
The former is a true segment since $D_i^2=D_i$, while the latter is necessarily a virtual trapezoid with the same character due to cancelations (cf. Figure 3).

\end{document}